\documentclass[12pt,a4paper]{article}
\usepackage{amsmath,amssymb,amsthm,hyperref}
\usepackage{graphicx}
\usepackage{version} 
\usepackage{color}
\usepackage{pstricks-add}
\usepackage[bf]{caption}     
\usepackage{psfrag}          
\usepackage{geometry}        
\usepackage{bm}  
\usepackage[utf8]{inputenc}
  \newtheorem{theorem}{Theorem}[section]

 \newtheorem{lemma}[theorem]{Lemma}

 {\theoremstyle{definition}
 \newtheorem{definition}[theorem]{Definition}
 \newtheorem{remark}[theorem]{Remark}
 
 \newtheorem{question}[theorem]{Question}
 
 }

\DeclareMathOperator{\diam}{diam}
\title{On series of translates of positive functions III}
\author{Zolt\'an Buczolich\thanks{
Research supported by the Hungarian National Research, Development and Innovation Office--NKFIH, Grant  124003. 
},
Department of Analysis, ELTE E\"otv\"os Lor\'and\\
University, P\'azm\'any P\'eter S\'et\'any 1/c, 1117 Budapest, Hungary\\
email: buczo@cs.elte.hu\\
{\tt www.cs.elte.hu/\hbox{$\sim$}buczo}\\
ORCID Id: 0000-0001-5481-8797\\
 \medskip\\
 Bal\'azs Maga\thanks{This author was supported by the \'UNKP-17-2 New National Excellence of the Hungarian Ministry of Human Capacities, and by the Hungarian National Research, Development and Innovation Office–NKFIH, Grant 124003.},
Department of Analysis, ELTE E\"otv\"os Lor\'and\\
University, P\'azm\'any P\'eter S\'et\'any 1/c, 1117 Budapest, Hungary\\
 email: magab@cs.elte.hu \\{\tt www.cs.elte.hu/\hbox{$\sim$}magab}\\
 \medskip\\
and \\
\\ \medskip\\
 G\'asp\'ar V\'ertesy\thanks{This author was supported by the Hungarian National Research, Development and Innovation Office–NKFIH, Grant 124749.
},
 Department of Analysis, ELTE E\"otv\"os Lor\'and\\
University, P\'azm\'any P\'eter S\'et\'any 1/c, 1117 Budapest, Hungary\\
email: vertesy.gaspar@gmail.com\
}
\date{\today}
\begin{document}
\maketitle

\medskip

\medskip
{\em Dedicated to the memory of Jean-Pierre Kahane}

\medskip


\begin{abstract}
 Suppose $\Lambda$ is a discrete infinite set of nonnegative real numbers.
We say that $ {\Lambda}$ is of type 1 if the series $s(x)=\sum_{\lambda\in\Lambda}f(x+\lambda)$
 satisfies a zero-one law. This means that for any non-negative measurable 
 $f: {\ensuremath {\mathbb R}}\to [0,+ {\infty})$ either the convergence set $C(f, {\Lambda})=\{x: s(x)<+ {\infty} \}= {\ensuremath {\mathbb R}}$ modulo sets of Lebesgue zero, or its complement the divergence set $D(f, {\Lambda})=\{x: s(x)=+ {\infty} \}= {\ensuremath {\mathbb R}}$ modulo sets of measure zero.
 If $ {\Lambda}$ is not of type 1 we say that $ {\Lambda}$ is of type 2.
 
 In this paper we show that  there is a universal
 $ {\Lambda}$ with gaps monotone decreasingly converging to zero
 such that for any open subset $G {\subset}  {\ensuremath {\mathbb R}}$ one can find 
 a characteristic function $f_{G}$
 such that $G {\subset} D(f_G, {\Lambda})$ and $C(f_G, {\Lambda})= {\ensuremath {\mathbb R}} {\setminus} G$ modulo sets of measure zero.
 
 We also consider the question whether  $C(f, {\Lambda})$ can contain
 non-degenerate intervals for continuous functions when $D(f, {\Lambda})$
 is of positive measure. 
 
 The above results answer some questions raised in a paper of Z. Buczolich,
 J-P. Kahane, and D. Mauldin.
  \end{abstract}


\section{Introduction}\label{*secintro}
This paper was written for the Kahane memorial volume of Analysis Mathematica.
 We selected a topic 
related to Jean-Pierre Kahane's work and decided to
answer some questions raised in  paper 
\cite{[BKM1]} by Z. Buczolich,
 J-P. Kahane, and D. Mauldin. 

This line of research was started in another  joint paper with Dan Mauldin 
\cite{[BM1]}. In that paper we considered a problem from 1970, originating
from the Diplomarbeit of
 Heinrich von Weizs\"aker \cite{[HW]}.

{\it Suppose $f:(0,+ {\infty})\to {\ensuremath {\mathbb R}}$ is a measurable function.
Is it true that
$\sum_{n=1}^{{\infty}}f(nx)$
either converges (Lebesgue) almost everywhere or diverges almost everywhere, i.e.
is there  a zero-one law for $\sum f(nx)$?}

This question also appeared in a paper of J. A. Haight \cite{[H1]}.

In \cite{[H1]} it was proved that there exists
a set
$H {\subset} (0, {\infty})$ of infinite measure, for which for all
 $x,y\in H,\  x\not=y$
the ratio $x/y$ is not an integer, and furthermore

{\it $(\dagger)$ for all $x>0$
$nx\not\in H$ if $n$ is sufficiently large.}

This implies that if
$f(x)=\chi_{H}(x)$,
the characteristic function of
 $H$ then
$\int_{0}^{{\infty}}f(x)dx= {\infty}$
and $\sum_{n=1}^{{\infty}}f(nx)< {\infty}$
everywhere.
\medskip

Lekkerkerker in \cite{[L]} started to study sets
with property
$(\dagger)$.

In  \cite{[BM1]}
we answered the Haight--Weizs\"aker problem.

\begin{theorem}\label{*HWth} There exists a
measurable function $f:(0,+ {\infty})\to \{0,1\}$ and
two nonempty intervals $I_{F}, \ I_{{\infty}} {\subset} [{1\over 2},1)$
such that for every $x\in I_{{\infty}}$ we have $\sum_{n=1}^{{\infty}}f(nx)=+ {\infty}$
and for almost every $x\in I_{F}$ we have $\sum_{n=1}^{{\infty}}f(nx)<+ {\infty}.$
The function $f$ is the characteristic function of an open set
$E$.\end{theorem}

Jean-Pierre Kahane was interested in this problem and soon after our paper had become available we started to receive faxes and emails from him.
This cooperation lead to papers \cite{[BKM1]} and \cite{[BKM2]}.

We considered a more general, additive version of the Haight--Weizs\"aker problem.
Since $\sum_{n=1}^{{\infty}}f(nx)=\sum_{n=1}^{{\infty}}f(e^{\log x+\log n})$,
that is using the function $h=f\circ   \exp$ defined on $ {\ensuremath {\mathbb R}}$ and $\Lambda=\{\log n:
n=1,2,... \}$   we were interested in almost everywhere
convergence questions of the series
$\sum_{{\lambda}\in {\Lambda}}h(x+ {\lambda})$.

Taking more general sets than $\Lambda=\{\log n:
n=1,2,... \}$  was also motivated by a paper, \cite{[H2]} of Haight.
He proved, using the original multiplicative notation of our problem that
if $ {\Lambda} {\subset}[0,+ {\infty})$ is an arbitrary countable set such that
its only accumulation point is $+ {\infty}$ then there exists a measurable set
$E {\subset}(0,+ {\infty})$ of infinite measure such that for all
 $x,y\in E$, $x\not=y$,
$x/y\not\in {\Lambda},$
and for a fixed
$x$ there exist
only finitely many $ {\lambda}\in {\Lambda}$ for which $ {\lambda} x\in
E$. This implies that choosing $f=\chi_{E}$
we have
$\sum_{{\lambda}\in {\Lambda}}f( {\lambda} x)< {\infty}$, but
$\int_{{\ensuremath {\mathbb R}}^+} f(x)dx= {\infty}.$ \medskip

Next we recall from \cite{[BKM1]} the definition of type 1 and type 2 sets.
Given $\Lambda$ an unbounded, infinite discrete set of nonnegative numbers, and
a measurable $f: {\ensuremath {\mathbb R}}\to [0,+ {\infty})$, we consider the sum
$$s(x)=\sum_{\lambda\in\Lambda}f(x+\lambda),$$
and the complementary subsets of $ {\ensuremath {\mathbb R}}$:
$$C=C(f,\Lambda)=\{x: s(x)< {\infty}\},\qquad
D=D(f,\Lambda)=\{x:s(x)= {\infty}\}.$$

\begin{definition}\label{def1} The set $\Lambda$ is of type 1 if, for every $f$, either
$C(f,\Lambda)= {\ensuremath {\mathbb R}}$ a.e. or $C(f,\Lambda)= {\emptyset}$ a.e. (or equivalently
$D(f,\Lambda)= {\emptyset}$ a.e. or $D(f,\Lambda)= {\ensuremath {\mathbb R}}$ a.e.). Otherwise, $\Lambda$
has type 2. \end{definition}

That is for type 1 sets we have a "zero-one" law for the almost everywhere
convergence properties of the series $\sum_{\lambda\in\Lambda}f(x+\lambda)$,
while for type 2 sets the situation is more complicated.

\begin{definition}\label{asydens} The unbounded, infinite  discrete set $\Lambda=\{{\lambda}_{1}, {\lambda}_{2},... \}$, $ {\lambda}_{1}< {\lambda}_{2}<...$ is asymptotically dense if $d_{n}= {\lambda}_{n}- {\lambda}_{n-1}\to 0$, or equivalently:
$$\forall a>0,\quad \lim_{x\to\infty}\#(\Lambda\cap [x,x+a])=\infty.$$
If $d_{n}$ tends to zero monotone decreasingly, we speak about
decreasing gap asymptotically dense sets.

If  $\Lambda$ is not asymptotically dense we say that it is asymptotically lacunary.
\end{definition}

We denote the non-negative continuous functions on $ {\ensuremath {\mathbb R}}$ by $C^{+}( {\ensuremath {\mathbb R}})$,
and if, in addition these functions tend to zero in $+ {\infty}$ they belong to
$C^{+}_{0}( {\ensuremath {\mathbb R}})$.

In \cite{[BKM1]} we gave some necessary and some sufficient conditions
for  a set $ {\Lambda}$ being of type 2. A complete characterization of type 2
sets is still unknown. We recall here 
from \cite{[BKM1]}
the theorem concerning the Haight--Weizs\"aker problem. This contains the additive version of the result of Theorem \ref{*HWth}
with some additional information.

\begin{theorem}\label{*BKMHWth} {The set $\Lambda=\{\log n:
n=1,2,... \}$ has type $2$. Moreover, for some $f\in
C_{0}^{+}( {\ensuremath {\mathbb R}}),$ $C(f, {\Lambda})$ has full measure on the half-line
$(0,\infty)$ and $D(f, {\Lambda})$ contains the half-line $(-\infty,0)$.
If for each $c, \int_c^{+\infty}e^yg(y)dy < +\infty$, then $C(g,\Lambda) =  {\ensuremath {\mathbb R}}$ a.e.
If $g\in C_{0}^{+}( {\ensuremath {\mathbb R}})$ and $C(g,\Lambda)$ is not of the first (Baire) category, then
$C(g, {\Lambda})= {\ensuremath {\mathbb R}}$ a.e. Finally, there is some $g\in C_{0}^+( {\ensuremath {\mathbb R}})$ such that $C(g,\Lambda) =  {\ensuremath {\mathbb R}}$ a.e. and
$\int_0^{+\infty}e^yg(y)dy = +\infty$.
}
\end{theorem}

As $ {\Lambda}$ used in the above theorem is a decreasing gap asymptotically dense
set and quite often it is much easier to construct examples with lacunary $ {\Lambda}$s,
in our paper we try to give examples with a decreasing gap asymptotically dense
$ {\Lambda}$.

One might believe that for type 2 $ {\Lambda}$s $C(f, {\Lambda})$, or $D(f, {\Lambda})$ are always half-lines
if they differ from $ {\ensuremath {\mathbb R}}$. 
Indeed in \cite{[BKM1]} we obtained results in this direction.
A number $t>0$ is called a translator of $ {\Lambda}$ if $( {\Lambda}+t) {\setminus}
 {\Lambda}$ is finite.
Condition $(*)$ is said to be satisfied if $T( {\Lambda})$, the
countable additive
semigroup of translators of $ {\Lambda}$, is dense in $ {\ensuremath {\mathbb R}}^{+}$.
We showed that condition $(*)$ implies that $C(f, {\Lambda})$ is
either $ {\emptyset}$, $ {\ensuremath {\mathbb R}}$,
or a right half-line modulo sets of measure zero. 

In \cite{[BM2]} we showed that this is not always the case.
For a given $ {\alpha}\in(0,1)$ and a sequence of natural numbers $n_{1}<n_{2}<...$ we
put
$ {\Lambda}^{{\alpha}^{k}}:=\cup_{k=1}^{{\infty}} {\Lambda}_{k}^{{\alpha}^{k}},$
$ {\Lambda}_{k}^{{\alpha}^{k}}
= {\alpha}^{k} {\ensuremath {\mathbb Z}}\cap [n_k,n_{k+1})$.

If $ {\alpha}=\frac{1}q$ for some $q\in \{2,3,...  \}$, then a slight
modification of
the proof of Theorem 1 of \cite{[BKM1]} shows that
$ {\Lambda}^{(\frac{1}q)^{k}}$
is of type 1 and condition
$(*)$ is satisfied. 

If $ {\alpha}\not\in  {\ensuremath {\mathbb Q}}$, then one can apply
Theorem 5
of \cite{[BKM1]} to show that $ {\Lambda}^{{\alpha}^{k}}$ is of type 2.

The difficult case is when $ {\alpha}=\frac{p}q$ with
$(p,q)=1,$ $p,q>1$, $p<q$.
In this case we showed that $ {\Lambda}^{{(\frac {p} {q} )}^{k}}$ is of type 2. 
In the cases
$ {\Lambda}^{{(\frac {p} {q} )}^{k}}$, $(p>1)$
condition $(*)$ is not satisfied and  we also
showed in \cite{[BM2]} that there exists a characteristic function $f$ such that
$C(f, {\Lambda})$ does not equal
$ {\emptyset}$, $ {\ensuremath {\mathbb R}}$, or a right half-line modulo sets of measure zero.
This
structure of $C(f, {\Lambda})$ had not been seen before our paper \cite{[BM2]}.

From the point of view of our current paper the following question
(QUESTION 2 in \cite{[BKM1]}) is the most relevant:
 
\begin{question}\label{*q1bkm1} Given open sets $G_1$ and $G_2$ when is it possible
to find $\Lambda$ and $f$ such that $C(f,\Lambda)$ contains $G_1$ and
$D(f,\Lambda)$ contains $G_2$?
\end{question}

It was remarked in \cite{[BKM1]}  that if the counting function
of $\Lambda,\  n(x)=\#\{{\Lambda}\cap [0,x]\}$ satisfies a condition of the type
$$
 \forall   \ell < 0\   \forall a\in  {{{\ensuremath {\mathbb R}}}} \ \ \limsup_{x \to
\infty}{n(x+  \ell+a) - n(x+a) \over n(x+  \ell) - n(x)} < +\infty  
$$
(as is the case for $\Lambda = \{\log n\}$) then either $C(f,\Lambda)$ has full
measure on $ {{{\ensuremath {\mathbb R}}}}$ or $C(f,\Lambda)$ does not contain any interval.

It was also mentioned in \cite{[BKM1]} that if $\Lambda$ is asymptotically lacunary
then it is possible to construct
$f \in C_0^+( {{{\ensuremath {\mathbb R}}}})$ such that both $C(f,\Lambda)$ and $D(f,\Lambda)$ have
interior points.

In this paper we give an almost complete answer to Question \ref{*q1bkm1}.
In Section \ref{*secdflg} we prove Theorem \ref{*thdflg}.
This theorem states that there is a universal decreasing gap asymptotically dense 
 $ {\Lambda}$ 
 such that for any open subset $G {\subset}  {\ensuremath {\mathbb R}}$ one can find 
 a characteristic function $f_{G}$
 such that $G {\subset} D(f_G, {\Lambda})$ and $C(f_G, {\Lambda})= {\ensuremath {\mathbb R}} {\setminus} G$ modulo sets of measure zero. We also show that one can also select a $g_{G}\in C_{0}^{+}( {\ensuremath {\mathbb R}})$
with similar properties. 

In Section \ref{*secsubi} we consider the question of subintervals in $C(f, {\Lambda})$
when $f\in C^{+}_{0}( {\ensuremath {\mathbb R}})$.
In Theorem \ref{*thdivG} we prove that there exists a universal asymptotically dense infinite discrete set $ {\Lambda}$
 such that for any open set $G {\subset}  {\ensuremath {\mathbb R}}$ one can select an
 $f_{G}\in C^{+}_{0}( {\ensuremath {\mathbb R}})$  such that $D(f_{G}, {\Lambda})=G.$ In this case there
 is no exceptional set of measure zero, $D(f_{G}, {\Lambda})$ equals $G$ exactly.
 On the other hand, $ {\Lambda}$ is not of decreasing gap. As Theorem \ref{*thcintl}
 shows it is impossible to find such a universal $ {\Lambda}$ with decreasing gaps.
 In Theorem \ref{*thcintl} we prove that 
 if $\Lambda$ is a decreasing gap asymptotically dense set,
 $f\in C^{+}( {\ensuremath {\mathbb R}})$ and $x$ is an interior point of $C(f,\Lambda)$   then 
$[x,+ {\infty})\cap D(f, {\Lambda})$ is of zero Lebesgue measure.

The example provided in Theorem \ref{*excint} demonstrates that there is a 
decreasing gap asymptotically dense $ {\Lambda}$ and an $f\in C^{+}_{0}( {\ensuremath {\mathbb R}})$ such that
$D(f, {\Lambda})$ and $C(f, {\Lambda})$ both contain interior points. Of course,
as Theorem \ref{*thcintl} shows the interior points of $D(f, {\Lambda})$ are
to the left of those of $C(f, {\Lambda})$.


\section{A universal decreasing gap asymptotically dense $ {\Lambda}$ set}\label{*secdflg}

Let $\mu$ denote the one-dimensional Lebesgue measure.

We denote by
$ {\ensuremath {\mathbb N}} := \{n\in {\ensuremath {\mathbb Z}} : n\ge 1\}$
the set of natural numbers.
For every $A,B\subset {\ensuremath {\mathbb R}}$ we put
$A+B:=\{a+b : a\in A \text{   and   } b\in B\}$ {and}
$A-B:=\{a-b : a\in A \text{   and   } b\in B\}.$

The integer, and fractional parts of $x\in  {\ensuremath {\mathbb R}}$ are denoted by $ {\lfloor} x  {\rfloor} $ and $\{x \}$,
respectively.

\begin{theorem}\label{*thdflg}
There is a strictly monotone increasing unbounded sequence $(\lambda_0,\lambda_1,\ldots)=\Lambda$ in $ {\mathbb {R}}$ such that  $\lambda_{n}-\lambda_{n-1}$ tends to $0$ monotone decreasingly,
that is $ {\Lambda}$ is a decreasing gap asymptotically dense set,
such that for every open set $G\subset {\mathbb {R}}$ there is a function $f_G: {\ensuremath {\mathbb R}}\to [0,+ {\infty})$ for which
\begin{equation}\label{tetelbeli}
\mu\left(\left\{x\notin G :  \sum_{n=0}^\infty  f_G(x+\lambda_n)=\infty\right\}\right)=0, \text{   and   }
 \end{equation}
\begin{equation}\label{*conv}
 \text{   $\sum_{n=0   }^\infty f_G(x+\lambda_n) =\infty$ for every $x\in G$,}
\end{equation} 
moreover
$f_G=\chi_{U_G}$ for a closed set $U_G\subset {\mathbb {R}}$.
By \eqref{tetelbeli} and \eqref{*conv} we have $D(f_G, {\Lambda})\supset G$, and 
$C(f_G, {\Lambda})= {\ensuremath {\mathbb R}} {\setminus} G$ modulo sets
of measure zero. 

One can also select a $g_G\in C_{0}^{+}( {\ensuremath {\mathbb R}})$ satisfying \eqref{tetelbeli} and
\eqref{*conv} instead of $f_G$. 
\end{theorem}

\begin{remark}
Observe that in the above theorem we construct a universal $ {\Lambda}$ and for this
set, depending on our choice of $G$ we can select a suitable $f_{G}$ such that
$D(f_{G}, {\Lambda})=G$ modulo sets of measure zero.
\end{remark}

\begin{proof}
Let 
\[
 {{\cal I}} := \{(j,k) : j\in {\ensuremath {\mathbb N}} \text{   and   } k \in  {\ensuremath {\mathbb Z}}\cap[0,2j\cdot 2^j)\} 
\] 
with the following lexicographical ordering: if $(j,k),(\widetilde{j},\widetilde{k})\in {{\cal I}}$ then
\[
(j,k) <_ {{\cal I}} (\widetilde{j},\widetilde{k}) \Leftrightarrow \big( j<\widetilde{j} \text{   or   } (j=\widetilde{j} \text{   and   } k<\widetilde{k})\big).
\]
Given $(j,k)\in {{\cal I}}$ we define its immediate successor $( {\hat {\jmath}} , {\hat {k}} )$
the following way:
let $ {\hat {\jmath}}  := j$ and $ {\hat {k}}  := k+1$ if $k < 2j\cdot 2^j-1$, and let $ {\hat {\jmath}}  := j+1$ and $ {\hat {k}}  := 0$ if $k = 2j\cdot 2^j-1$. It is clear that starting with $(1,0)$ by repeated application of taking
the immediate successor we can enumerate $ {{\cal I}}$ and hence we will be able to do
induction on $ {{\cal I}}$. We will also introduce the operation of taking the predecessor
of $(j,k)\not=(1,0)$ which will be denoted by $( {\check {\jmath}}, {\check {k}})$
and which is defined by the property $(\hat{{\check {\jmath}}},\hat{{\check {k}}})=(j,k)$.

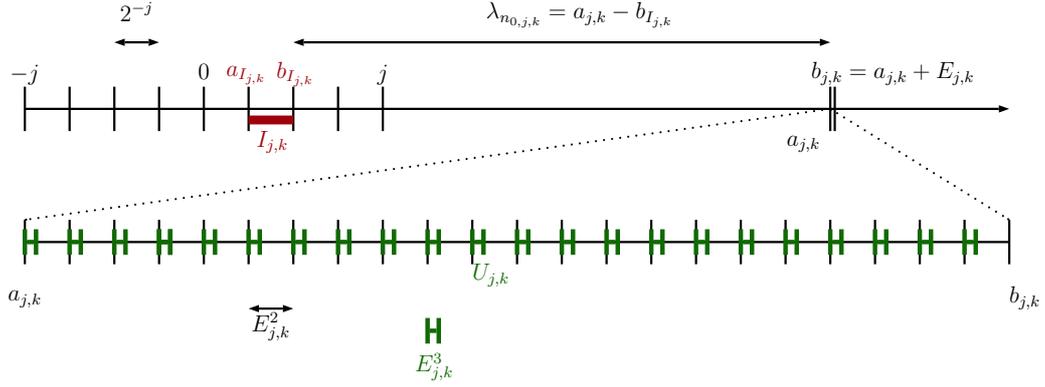
\begin{figure}[h]
\centering{
\resizebox{1.0\textwidth}{!}{
%
\psscalebox{1.0 1.0} 
{
\begin{pspicture}(0,-3.3392189)(19.57,3.3392189)
\definecolor{colour0}{rgb}{0.62352943,0.007843138,0.043137256}
\definecolor{colour5}{rgb}{0.0627451,0.41568628,0.007843138}
\psline[linecolor=black, linewidth=0.04, arrowsize=0.05291667cm 2.0,arrowlength=1.4,arrowinset=0.0]{->}(0.81453127,1.5041211)(18.414532,1.5041211)
\psline[linecolor=black, linewidth=0.04](0.81453127,1.904121)(0.81453127,1.1041211)
\psline[linecolor=black, linewidth=0.04](1.6145313,1.904121)(1.6145313,1.1041211)
\psline[linecolor=black, linewidth=0.04](2.4145312,1.904121)(2.4145312,1.1041211)
\psline[linecolor=black, linewidth=0.04](3.2145312,1.904121)(3.2145312,1.1041211)
\psline[linecolor=black, linewidth=0.04](4.014531,1.904121)(4.014531,1.1041211)
\psline[linecolor=black, linewidth=0.04](4.8145313,1.904121)(4.8145313,1.1041211)
\psline[linecolor=black, linewidth=0.04](5.614531,1.904121)(5.614531,1.1041211)
\psline[linecolor=black, linewidth=0.04](6.414531,1.904121)(6.414531,1.1041211)
\psline[linecolor=black, linewidth=0.04](7.2145314,1.904121)(7.2145314,1.1041211)
\psline[linecolor=black, linewidth=0.04](15.214531,1.904121)(15.214531,1.1041211)
\psline[linecolor=black, linewidth=0.04](15.294531,1.904121)(15.294531,1.1041211)
\psline[linecolor=colour0, linewidth=0.16](4.8145313,1.3041211)(5.614531,1.3041211)
\rput[bl](4.954531,0.7041211){\textcolor{colour0}{$I_{j,k}$}}
\rput[t](0.81453127,2.304121){$-j$}
\rput[t](7.2145314,2.304121){$j$}
\rput[t](4.014531,2.304121){$0$}
\rput[t](14.734531,1.024121){$a_{j,k}$}
\rput[t](16.334532,2.344121){$b_{j,k}=a_{j,k}+E_{j,k}$}
\psline[linecolor=black, linewidth=0.04, linestyle=dotted, dotsep=0.10583334cm, dotsize=0.07055555cm 2.0,dotsize=0.07055555cm 2.0]{cc-cc}(15.214531,1.5041211)(0.81453127,-0.4958789)
\psline[linecolor=black, linewidth=0.04, linestyle=dotted, dotsep=0.10583334cm, dotsize=0.07055555cm 2.0,dotsize=0.07055555cm 2.0]{cc-cc}(15.214531,1.5041211)(18.414532,-0.4958789)
\psline[linecolor=black, linewidth=0.04](0.81453127,-0.8958789)(18.414532,-0.8958789)
\psline[linecolor=black, linewidth=0.04](0.81453127,-0.4958789)(0.81453127,-1.2958789)
\psline[linecolor=black, linewidth=0.04](18.414532,-0.4958789)(18.414532,-1.2958789)
\psline[linecolor=black, linewidth=0.04](1.6145313,-0.4958789)(1.6145313,-1.2958789)
\psline[linecolor=black, linewidth=0.04](2.4145312,-0.4958789)(2.4145312,-1.2958789)
\psline[linecolor=black, linewidth=0.04](3.2145312,-0.4958789)(3.2145312,-1.2958789)
\psline[linecolor=black, linewidth=0.04](4.014531,-0.4958789)(4.014531,-1.2958789)
\psline[linecolor=black, linewidth=0.04](4.8145313,-0.4958789)(4.8145313,-1.2958789)
\psline[linecolor=black, linewidth=0.04](5.614531,-0.4958789)(5.614531,-1.2958789)
\psline[linecolor=black, linewidth=0.04](6.414531,-0.4958789)(6.414531,-1.2958789)
\psline[linecolor=black, linewidth=0.04](7.2145314,-0.4958789)(7.2145314,-1.2958789)
\psline[linecolor=black, linewidth=0.04](8.014531,-0.4958789)(8.014531,-1.2958789)
\psline[linecolor=black, linewidth=0.04](8.814531,-0.4958789)(8.814531,-1.2958789)
\psline[linecolor=black, linewidth=0.04](9.6145315,-0.4958789)(9.6145315,-1.2958789)
\psline[linecolor=black, linewidth=0.04](10.414532,-0.4958789)(10.414532,-1.2958789)
\psline[linecolor=black, linewidth=0.04](11.214531,-0.4958789)(11.214531,-1.2958789)
\psline[linecolor=black, linewidth=0.04](12.014531,-0.4958789)(12.014531,-1.2958789)
\psline[linecolor=black, linewidth=0.04](12.814531,-0.4958789)(12.814531,-1.2958789)
\psline[linecolor=black, linewidth=0.04](13.6145315,-0.4958789)(13.6145315,-1.2958789)
\psline[linecolor=black, linewidth=0.04](14.414532,-0.4958789)(14.414532,-1.2958789)
\psline[linecolor=black, linewidth=0.04](15.214531,-0.4958789)(15.214531,-1.2958789)
\psline[linecolor=black, linewidth=0.04](16.014532,-0.4958789)(16.014532,-1.2958789)
\psline[linecolor=black, linewidth=0.04](16.814531,-0.4958789)(16.814531,-1.2958789)
\psline[linecolor=black, linewidth=0.04](17.61453,-0.4958789)(17.61453,-1.2958789)
\rput[t](18.69453,-1.6958789){$b_{j,k}$}
\rput[t](0.81453127,-1.7758789){$a_{j,k}$}
\rput[t](2.8145313,3.4241211){$2^{-j}$}
\psline[linecolor=black, linewidth=0.04, arrowsize=0.05291667cm 2.0,arrowlength=1.4,arrowinset=0.0]{<->}(2.4145312,2.704121)(3.2145312,2.704121)
\rput[t](5.2145314,-2.175879){$E_{j,k}^2$}
\psline[linecolor=black, linewidth=0.04, arrowsize=0.05291667cm 2.0,arrowlength=1.4,arrowinset=0.0]{<->}(4.8145313,-2.0958788)(5.614531,-2.0958788)
\psline[linecolor=colour5, linewidth=0.08, tbarsize=0.1411111cm 4.0]{|*-|*}(3.2145312,-0.8958789)(3.4145312,-0.8958789)
\psline[linecolor=colour5, linewidth=0.08, tbarsize=0.1411111cm 4.0]{|*-|*}(0.81453127,-0.8958789)(1.0145313,-0.8958789)
\psline[linecolor=colour5, linewidth=0.08, tbarsize=0.1411111cm 4.0]{|*-|*}(1.6145313,-0.8958789)(1.8145312,-0.8958789)
\psline[linecolor=colour5, linewidth=0.08, tbarsize=0.1411111cm 4.0]{|*-|*}(2.4145312,-0.8958789)(2.6145313,-0.8958789)
\psline[linecolor=colour5, linewidth=0.08, tbarsize=0.1411111cm 4.0]{|*-|*}(3.2145312,-0.8958789)(3.4145312,-0.8958789)
\psline[linecolor=colour5, linewidth=0.08, tbarsize=0.1411111cm 4.0]{|*-|*}(4.014531,-0.8958789)(4.2145314,-0.8958789)
\psline[linecolor=colour5, linewidth=0.08, tbarsize=0.1411111cm 4.0]{|*-|*}(4.8145313,-0.8958789)(5.014531,-0.8958789)
\psline[linecolor=colour5, linewidth=0.08, tbarsize=0.1411111cm 4.0]{|*-|*}(5.614531,-0.8958789)(5.8145313,-0.8958789)
\psline[linecolor=colour5, linewidth=0.08, tbarsize=0.1411111cm 4.0]{|*-|*}(6.414531,-0.8958789)(6.614531,-0.8958789)
\psline[linecolor=colour5, linewidth=0.08, tbarsize=0.1411111cm 4.0]{|*-|*}(7.2145314,-0.8958789)(7.414531,-0.8958789)
\psline[linecolor=colour5, linewidth=0.08, tbarsize=0.1411111cm 4.0]{|*-|*}(8.014531,-0.8958789)(8.214531,-0.8958789)
\psline[linecolor=colour5, linewidth=0.08, tbarsize=0.1411111cm 4.0]{|*-|*}(8.814531,-0.8958789)(9.014531,-0.8958789)
\psline[linecolor=colour5, linewidth=0.08, tbarsize=0.1411111cm 4.0]{|*-|*}(9.6145315,-0.8958789)(9.814531,-0.8958789)
\psline[linecolor=colour5, linewidth=0.08, tbarsize=0.1411111cm 4.0]{|*-|*}(10.414532,-0.8958789)(10.6145315,-0.8958789)
\psline[linecolor=colour5, linewidth=0.08, tbarsize=0.1411111cm 4.0]{|*-|*}(11.214531,-0.8958789)(11.414532,-0.8958789)
\psline[linecolor=colour5, linewidth=0.08, tbarsize=0.1411111cm 4.0]{|*-|*}(12.014531,-0.8958789)(12.214531,-0.8958789)
\psline[linecolor=colour5, linewidth=0.08, tbarsize=0.1411111cm 4.0]{|*-|*}(12.814531,-0.8958789)(13.014531,-0.8958789)
\psline[linecolor=colour5, linewidth=0.08, tbarsize=0.1411111cm 4.0]{|*-|*}(13.6145315,-0.8958789)(13.814531,-0.8958789)
\psline[linecolor=colour5, linewidth=0.08, tbarsize=0.1411111cm 4.0]{|*-|*}(14.414532,-0.8958789)(14.6145315,-0.8958789)
\psline[linecolor=colour5, linewidth=0.08, tbarsize=0.1411111cm 4.0]{|*-|*}(15.214531,-0.8958789)(15.414532,-0.8958789)
\psline[linecolor=colour5, linewidth=0.08, tbarsize=0.1411111cm 4.0]{|*-|*}(16.014532,-0.8958789)(16.21453,-0.8958789)
\psline[linecolor=colour5, linewidth=0.08, tbarsize=0.1411111cm 4.0]{|*-|*}(16.814531,-0.8958789)(17.014532,-0.8958789)
\psline[linecolor=colour5, linewidth=0.08, tbarsize=0.1411111cm 4.0]{|*-|*}(17.61453,-0.8958789)(17.814531,-0.8958789)
\rput[bl](8.814531,-1.6958789){\textcolor{colour5}{$U_{j,k}$}}
\psline[linecolor=colour5, linewidth=0.08, tbarsize=0.1411111cm 4.0]{|*-|*}(8.014531,-2.495879)(8.214531,-2.495879)
\rput[t](8.134531,-2.8958788){\textcolor{colour5}{$E_{j,k}^3$}}
\rput[bl](5.3145313,1.904121){\textcolor{colour0}{$b_{I_{j,k}}$}}
\rput[bl](4.414531,1.9441211){\textcolor{colour0}{$a_{I_{j,k}}$}}
\psline[linecolor=black, linewidth=0.04, arrowsize=0.05291667cm 2.0,arrowlength=1.4,arrowinset=0.0]{<->}(5.614531,2.704121)(15.214531,2.704121)
\rput[t](10.734531,3.4241211){$\lambda_{n_{0,j,k}}=a_{j,k}-b_{I_{j,k}}$}
\end{pspicture}
}}}
\caption{Definition of $I_{j,k}$ and $U_{j,k}$} \label{*figdefGU}
\end{figure}

For every $(j,k)\in  {{\cal I}}$ let
\begin{equation*}\label{I_{j,k}}
I_{j,k} := \left[j-(k+1)2^{-j}, j-k2^{-j}\right]=[a_{I_{j,k}},b_{I_{j,k}}].
\end{equation*}
In \eqref{*Ujk} a set $U_{j,k}$ will be defined such that with a properly selected $\Lambda$ we have
\begin{equation}\label{I_{j,k} fedese}
I_{j,k} \subset U_{j,k}-\Lambda = \left\{x\in {\ensuremath {\mathbb R}} :   \exists n\in {\ensuremath {\mathbb N}}\cup\{0\} \text{   such that   } x+\lambda_n\in U_{j,k} \right\} \text{   and   }
\end{equation}
\begin{equation}\label{kompl}
\begin{gathered}
\mu ( \{x\in [-j,j] :   \exists \text{   infinitely    
 many $(j^*,k^*)\in {{\cal I}}$} \\
\text{   for which   } x \in
 \left(U_{j^*,k^*}-\Lambda\right)\setminus I_{j^*,k^*} \} ) = 0.
\end{gathered}
\end{equation}
Let $G$ be an arbitrary open subset of $ {\ensuremath {\mathbb R}}$ and let
\[
U_G :=  {\bigcup}\left\{U_{j^*,k^*} : (j^*,k^*)\in {{\cal I}} \text{   and   } I_{j^*,k^*}\subset G \right\}. 
\]
Put
\begin{equation}\label{def f_G}
f_G(x):=
\begin{cases}
1 & \text{   if   }x\in U_G \\  
0 & \text{   else   }.
\end{cases}
\end{equation}
We will prove that  $\Lambda$ and $f_G$ satisfy the conditions of the theorem.

Now we define the sets $U_{j,k}$.
 Before doing this we recall and introduce some notation. 
For every $(j,k)\in {{\cal I}}$ let 
\begin{itemize}
\item $a_{I_{j,k}}:=j-(k+1)\cdot 2^{-j}$ (that is $a_{I_{j,k}}$ is the left endpoint of $I_{j,k}$),
\item $b_{I_{j,k}}:=j-k\cdot 2^{-j}$ (that is $b_{I_{j,k}}$ is the right endpoint of $I_{j,k}$),
\item $E_{j,k}:=2^{-2j\cdot 2^j-k}$,
\item $a_{j,k}:=2^{2j\cdot 2^j+k}$,
\item $b_{j,k}:=a_{j,k}+E_{j,k}$.
\end{itemize}
See Figure \ref{*figdefGU}. This and the other figure in this paper are to illustrate
concepts and they are not drawn to illustrate a certain step, for example with a fixed $j$
of our construction.

Let
\begin{equation}\label{*Ujk}
U_{j,k} :=  {\bigcup}_{i=0}^{E^{-1}_{j,k}-1} [a_{j,k}+iE^2_{j,k},a_{j,k}+iE^2_{j,k}+E^3_{j,k}] \subset [a_{j,k},b_{j,k}].
\end{equation}

Next we prove a useful lemma:

\begin{lemma}\label{useful}
For every $(j,k)\in {{\cal I}}$ we have
\begin{equation}\label{*lem23}
 a_{j,k} \le \frac{a_{{\hat {\jmath}}, {\hat {k}}}}{2} \text{   and   } E_{j,k} \ge 2E_{{\hat {\jmath}}, {\hat {k}}},
\end{equation}
moreover,
\begin{equation}\label{*lem23b}
E_{j,k}/2 \text{   is an integer multiple of   }E_{{\hat {\jmath}}, {\hat {k}}}.
\end{equation}
\end{lemma}
\begin{proof}
It is enough to prove  \eqref{*lem23} for $a_{j,k}$ as $E_{j,k}=a^{-1}_{j,k}$.

First suppose that $k < 2j\cdot 2^j-1 $, then $ {\hat {\jmath}}=j$, $ {\hat {k}}=k+1$ and
\begin{equation}\label{*ajk231}
a_{j,k} = 2^{2j\cdot 2^j+k} = \frac{2^{2j\cdot 2^j+(k+1)}}{2} = \frac{a_{{\hat {\jmath}} , {\hat {k}}}}{2}. 
\end{equation}
If $k = 2j\cdot 2^j-1 $ then $ {\hat {\jmath}}=j+1$, $ {\hat {k}}=0$ and
\begin{equation}\label{*ajk232}
\begin{split}
a_{j,k} &= 2^{2j\cdot 2^j+k} = 2^{2j\cdot 2^j+2j\cdot 2^j-1} = 2^{4j2^j-1} = \frac{2^{2(j+1)\cdot 2^{(j+1)}}}{2^{2\cdot2^{j+1}+1}} = \frac{a_{{\hat {\jmath}} , {\hat {k}}}}{2^{2\cdot2^{j+1}+1}}. 
\end{split}
\end{equation}
 Using $E_{j,k}=a^{-1}_{j,k}$ from \eqref{*ajk231} and \eqref{*ajk232}
it follows that \eqref{*lem23b} holds.
\end{proof}

Next we turn to the definition of $ {\Lambda}$.

During the definition of $ {\Lambda}$ we will use the notation $d_n:=\lambda_{n}-\lambda_{n-1}$, 
in fact,
 often we will define $d_{n}$ and that will provide the value of $ {\lambda}_{n}$ given
the already defined $ {\lambda}_{n-1}$. Let $\lambda_0:=a_{{1,0}}-b_{I_{1,0}}$
 and $n_{0,1,0}=0.$

Suppose that for a $(j,k)\in  {{\cal I}}$ we have already defined $n_{0,j,k}$
and $ {\lambda}_{n}$ for $n\leq n_{0,j,k}$,
 $\lambda_{n_{0,j,k}}=a_{j,k}-b_{I_{j,k}}$ and $d_{n_{0,j,k}}/E^2_{j,k}$ is a positive integer (or $n_{0,j,k}=0$). Now we need to do our next step to define these
 objects for $( {\hat {\jmath}} , {\hat {k}} ).$

\textbf{Step $( {\hat {\jmath}} , {\hat {k}} )$.}  Let $n_{1,j,k} := n_{0,j,k}+2^{-j}E^{-2}_{j,k}+2E^{-1}_{j,k}$. For every integer $n\in[n_{0,j,k}+1,n_{1,j,k}]$ let $d_n := E^2_{j,k}-E^3_{j,k}$. Thus we have
\begin{equation}\label{lambda_{n_{1,j,k}}}
\begin{split}
\lambda_{n_{1,j,k}} &= \lambda_{n_{0,j,k}} + (2^{-j}E^{-2}_{j,k}+2E^{-1}_{j,k})(E^2_{j,k}-E^3_{j,k}) \\
&= a_{j,k} - b_{I_{j,k}} + 2^{-j} - 2^{-j}E_{j,k} + 2E_{j,k} - 2E^2_{j,k} \\
&= a_{j,k} - a_{I_{j,k}} + 2E_{j,k} - 2^{-j}E_{j,k} - 2E^2_{j,k} \\
&=b_{j,k} - a_{I_{j,k}} + E_{j,k} - 2^{-j}E_{j,k} - 2E^2_{j,k} \ge b_{j,k} - a_{I_{j,k}}
\end{split}
\end{equation}
and (from the second row of \eqref{lambda_{n_{1,j,k}}})
\begin{equation}\label{hasznos}
\begin{split}
\lambda_{n_{1,j,k}} &= a_{j,k} - b_{I_{j,k}} + 2^{-j} - 2^{-j}E_{j,k} + 2E_{j,k} - 2E^2_{j,k} < a_{j,k} - b_{I_{j,k}} + 1.
\end{split}
\end{equation}

Since $a_{j,k}-a_{I_{j,k}}=2^{2j\cdot 2^{j}+k}-(j-k\cdot 2^{-j})$
and $2^{-j}E_{j,k} $ are both
 integer multiples of  $E^2_{j,k}=(2^{-2j\cdot 2^j-k})^{2}$ 
 from the third row of \eqref{lambda_{n_{1,j,k}}} we obtain that
\begin{equation}\label{*intma}
\text{   $\lambda_{n_{1,j,k   }}$ is an integer multiple of }E^2_{j,k}.
\end{equation}

By Lemma \ref{useful} and \eqref{hasznos} we have
\[
a_{{\hat {\jmath}} , {\hat {k}}}-b_{I_{{\hat {\jmath}} , {\hat {k}}}} \ge 2a_{j,k} - (j+1) \ge a_{j,k} + j+1 > a_{j,k} - b_{I_{j,k}} + 1 > \lambda_{n_{1,j,k}}. 
\]

 We set 
\begin{equation}\label{*nojkkk}
 n_{0, {\hat {\jmath}} , {\hat {k}}}=n_{1,j,k}+\frac{a_{{\hat {\jmath}} , {\hat {k}}}-b_{I_{{\hat {\jmath}} , {\hat {k}}}}-\lambda_{n_{1,j,k}}}{2^{-1}E^2_{j,k}} 
\end{equation}
 and 
\begin{equation}\label{*dnkk}
\text{   $d_n=E^2_{j,k   }/2$ for every integer $n\in (n_{1,j,k}, n_{0, {\hat {\jmath}} , {\hat {k}}}]$.}
\end{equation}
 We obtain by \eqref{*nojkkk} 
\[
\lambda_{n_{0, {\hat {\jmath}} , {\hat {k}}}} = \lambda_{n_{1,j,k}} + \frac{(n_{0, {\hat {\jmath}} , {\hat {k}}}-n_{1,j,k})E^2_{j,k}}{2} = \lambda_{n_{1,j,k}} + a_{{\hat {\jmath}} , {\hat {k}}} - b_{I_{{\hat {\jmath}} , {\hat {k}}}} - \lambda_{n_{1,j,k}} = a_{{\hat {\jmath}} , {\hat {k}}} - b_{I_{{\hat {\jmath}} , {\hat {k}}}},
\]
and by \eqref{*lem23b}, $d_{n_{0, {\hat {\jmath}} , {\hat {k}}}} = E^2_{j,k}/2$ is an integer multiple of $E^2_{{\hat {\jmath}} , {\hat {k}}}$, hence \eqref{*intma} implies that
\begin{equation}\label{*intmb}
\text{   $\lambda_{n   }$ is an integer multiple of $E^2_{{\hat {\jmath}} , {\hat {k}}}$ for $n\in (n_{1,j,k}, n_{0, {\hat {\jmath}} , {\hat {k}}}]$.}
\end{equation}

Thus we can proceed to the next step. 
By repeating this procedure
 we can  carry out the above steps for all $(j,k)\in  {{\cal I}}$
and hence we can define $ {\Lambda}$.

Now we prove \eqref{I_{j,k} fedese}. We fix $(j,k)$ and 
choose an arbitrary point $x$ from $I_{j,k}$. 
Let $n_x$ denote the smallest integer for which 
\begin{equation}\label{n_x}
x+\lambda_{n_x} > a_{j,k}. 
\end{equation}
Put $n'_x:=n_x+\left  \lfloor\frac{x+\lambda_{n_x}-a_{j,k}}{E^3_{j,k}}\right  \rfloor$. 

We have $x\in I_{j,k} {\subset} [-j,j]$. From $x+ {\lambda}_{n_{0,j,k}}=x+a_{j,k}-b_{I_{j,k}}$
it follows that
\begin{equation}\label{*xlnojk}
x+ {\lambda}_{n_{0,j,k}}-a_{j,k}=x-b_{I_{j,k}}\leq 0.
\end{equation}
Therefore, $n_{x}> n_{0,j,k}$ and hence 
\begin{equation}\label{*19e*}
\text{$d_n \le d_{n_{0,j,k}+1} = E^2_{j,k}-E^3_{j,k}$ for every $n\in [n_x,\infty).$}
\end{equation}

By minimality of $n_{x}$ we have 
\begin{equation}\label{*nxmin}
x+\lambda_{n_x}-a_{j,k}\leq d_{n_{x}} \leq E^2_{j,k}-E^3_{j,k}.
\end{equation}

Next we will show that $x+\lambda_{n'_x}\in U_{j,k}$. 
 Using \eqref{*19e*}
\begin{equation}\label{szamolas2}
0 \le \left  \lfloor\frac{x+\lambda_{n_x}-a_{j,k}}{E^3_{j,k}}\right  \rfloor \le  \frac{d_{n_x}}{E^3_{j,k}} \le  \frac{E^2_{j,k}-E^3_{j,k}}{E^3_{j,k}} = E^{-1}_{j,k}-1.
\end{equation}
We also infer
\begin{equation}
\begin{split}
\label{szamolas3}
x+\lambda_{n'_x} &= x + \lambda_{n_x} + \sum_{n\in (n_x,n'_x]} d_n \le x + \lambda_{n_x} + \left  \lfloor\frac{x+\lambda_{n_x}-a_{j,k}}{E^3_{j,k}}\right  \rfloor(E^2_{j,k}-E^3_{j,k}) \\ 
&= a_{j,k} + (x+\lambda_{n_x}-a_{j,k}) + \left  \lfloor\frac{x+\lambda_{n_x}-a_{j,k}}{E^3_{j,k}}\right  \rfloor(E^2_{j,k}-E^3_{j,k}) \\
&= a_{j,k} + \left  \lfloor\frac{x+\lambda_{n_x}-a_{j,k}}{E^3_{j,k}}\right  \rfloor E^2_{j,k} + E^3_{j,k} \left\{\frac{x+\lambda_{n_x}-a_{j,k}}{E^3_{j,k}}\right\}\\
&\text{   using \eqref{szamolas2}}
\\
&\le a_{j,k} + (E^{-1}_{j,k}-1)E^2_{j,k} + E^3_{j,k} \le a_{j,k} + E_{j,k} = b_{j,k}.
\end{split}
\end{equation}
 From \eqref{lambda_{n_{1,j,k}}} and  \eqref{szamolas3}  we obtain
\[
\lambda_{n'_x} \le b_{j,k}-x \le b_{j,k}-a_{I_{j,k}} \le \lambda_{n_{1,j,k}},
\]
hence $n_x,n'_x \le n_{1,j,k}$, which means that $d_n=E^2_{j,k}-E^3_{j,k}$ for every $n\in (n_x,n'_x]$. This implies that the first inequality in \eqref{szamolas3} is, 
in fact an equality, that is
\begin{equation}\label{szamolas3'}
\begin{split}
x+\lambda_{n'_x} = a_{j,k} + \left  \lfloor\frac{x+\lambda_{n_x}-a_{j,k}}{E^3_{j,k}}\right  \rfloor E^2_{j,k} + E^3_{j,k}\left\{\frac{x+\lambda_{n_x}-a_{j,k}}{E^3_{j,k}}\right\}
.\end{split}
\end{equation}
Using \eqref{szamolas2} and \eqref{szamolas3'} we can see that there exists an integer $i=\left  \lfloor\frac{x+\lambda_{n_x}-a_{j,k}}{E^3_{j,k}}\right  \rfloor\in[0,E^{-1}_{j,k}-1]$ such that 
\[
a_{j,k}+iE^2_{j,k} \le x+\lambda_{n'_x} \le a_{j,k}+iE^2_{j,k}+E^3_{j,k}
\]
that is $x+\lambda_{n'_x}\in U_{j,k}$, which implies \eqref{I_{j,k} fedese}.

We continue with the proof of \eqref{kompl}. 
Suppose $( {\check {\jmath}}, {\check {k}}),(j,k),( {\hat {\jmath}} , {\hat {k}} )\in {{\cal I}}$. 
Then they are strictly monotone increasing in this order and are adjacent in
the lexicographical ordering of  $ {{\cal I}}$. 
We have by Lemma \ref{useful} and the third row of \eqref{lambda_{n_{1,j,k}}} 
\begin{equation}
\begin{split}
\label{szamolas4}
j + \lambda_{n_{1, {\check {\jmath}}, {\check {k}}}} &= j + a_{{\check {\jmath}}, {\check {k}}} - a_{I_{{\check {\jmath}}, {\check {k}}}} + 2E_{{\check {\jmath}}, {\check {k}}} - 2^{- {\check {\jmath}}}E_{{\check {\jmath}}, {\check {k}}} - 2E^2_{{\check {\jmath}}, {\check {k}}}\\
 &< a_{{\check {\jmath}}, {\check {k}}} + 2j + 1 \le 2a_{{\check {\jmath}}, {\check {k}}} \le a_{j,k},
\end{split}
\end{equation}
that is $U_{j,k}-\lambda_{n_{1, {\check {\jmath}}, {\check {k}}}}$ is to the right of $j$. By \eqref{*intmb}, $\lambda_n/E^2_{j,k}$ is an integer for every $n\in (n_{1, {\check {\jmath}}, {\check {k}}},n_{0,j,k}]$. Therefore, \eqref{szamolas4} implies that
\begin{equation}\label{szamolas5}
\begin{split}
B_{j,k} :&= [b_{j,k}-\lambda_{n_{0,j,k}},j] \cap (U_{j,k}-\Lambda) \\
 &= [b_{j,k}-\lambda_{n_{0,j,k}},j] \cap \left(U_{j,k}-\{\lambda_n : n\in (n_{1, {\check {\jmath}}, {\check {k}}},n_{0,j,k}]\}\right)\\
&  {\subset} [b_{j,k}-\lambda_{n_{0,j,k}},j] \cap \bigcup_{i\in {\ensuremath {\mathbb Z}}} [iE^2_{j,k},iE^2_{j,k}+E^3_{j,k}].
\end{split}
\end{equation}
Similarly, by using \eqref{*lem23}
\begin{equation}\label{szamolas6}
\begin{split}
-j + \lambda_{n_{0, {\hat {\jmath}} , {\hat {k}}}} &= -j + a_{{\hat {\jmath}} , {\hat {k}}} - b_{I_{{\hat {\jmath}} , {\hat {k}}}}  > a_{{\hat {\jmath}}, {\hat {k}}} - (2j+1) \\
&\ge 2a_{j,k} - (2j+1) \ge a_{j,k} + E_{j,k} = b_{j,k},
\end{split}
\end{equation}
that is $U_{j,k}-\lambda_{n_{0, {\hat {\jmath}} , {\hat {k}}}}$ is to the left of $-j$. Since 
by \eqref{*intma} and \eqref{*dnkk}
$\lambda_n/\left(E^2_{j,k}/2\right)$ is an integer for every $n\in [n_{1,j,k},n_{0, {\hat {\jmath}} , {\hat {k}}}]$, \eqref{szamolas6} implies that
\begin{equation}\label{szamolas7}
\begin{split}
A_{j,k} :&= [-j,a_{j,k}-\lambda_{n_{1,j,k}}] \cap (U_{j,k}-\Lambda) \\
 &= [-j,a_{j,k}-\lambda_{n_{1,j,k}}] \cap \left(U_{j,k}-\{\lambda_n : n\in [n_{1,j,k},n_{0, {\hat {\jmath}} , {\hat {k}}}]\}\right)\\
& {\subset} [-j,a_{j,k}-\lambda_{n_{1,j,k}}] \cap  {\bigcup}_{i\in {\ensuremath {\mathbb Z}}} [iE^2_{j,k}/2,iE^2_{j,k}/2+E^3_{j,k}].
\end{split}
\end{equation}

We want to estimate the following expression from above: 
\begin{equation}\label{kompl_megint}
\begin{gathered}
\mu\left( [-j,j] \cap (U_{j,k}-\Lambda)\setminus I_{j,k} \right) \\
 \le \mu\left( A_{j,k} \cup [a_{j,k}-\lambda_{n_{1,j,k}},a_{I_{j,k}}] \cup [b_{I_{j,k}},b_{j,k}-\lambda_{n_{0,j,k}}] \cup B_{j,k} \right).
\end{gathered}
\end{equation}
By \eqref{szamolas5} and \eqref{szamolas7} we have
\begin{equation}\label{szamolas8}
\begin{split}
&\mu\left(A_{j,k} \cup B_{j,k}\right) \\
&\le \mu\left( [-j,j] \cap \Big ( {\bigcup}_{i\in {\ensuremath {\mathbb Z}}} [iE^2_{j,k}/2,iE^2_{j,k}/2+E^3_{j,k}]  \Big )\right) \\
&= E^3_{j,k} \dfrac{2j}{E^2_{j,k}/2}   = 4j\cdot E_{j,k},
\end{split}
\end{equation}
and  using the third row of \eqref{lambda_{n_{1,j,k}}}
\begin{equation}\label{szamolas9}
\begin{split}
\mu\left( [a_{j,k}-\lambda_{n_{1,j,k}},a_{I_{j,k}}] \right) &= a_{I_{j,k}}-\left( a_{j,k}-(a_{j,k} - a_{I_{j,k}} + 2E_{j,k} - 2^{-j}E_{j,k} - 2E^2_{j,k}) \right) \\
&= 2E_{j,k} - 2^{-j}E_{j,k} - 2E^2_{j,k} \le 2E_{j,k}.
\end{split}
\end{equation}
Moreover,
\begin{equation}\label{szamolas10}
\begin{split}
\mu[b_{I_{j,k}},b_{j,k}-\lambda_{n_{0,j,k}}] &= b_{j,k}-(a_{j,k}-b_{I_{j,k}})-b_{I_{j,k}} = b_{j,k}-a_{j,k} = E_{j,k}.
\end{split}
\end{equation}
Writing \eqref{szamolas8}, \eqref{szamolas9} and \eqref{szamolas10} into \eqref{kompl_megint} yields
\begin{equation}\label{kompl_becsles1}
\mu\left( [-j,j] \cap (U_{j,k}-\Lambda)\setminus I_{j,k} \right) \le 
(4j+3)\cdot E_{j,k}.
\end{equation}
Thus
\begin{equation}
\label{kompl_becsles2}
\sum_{(j^*,k^*)\in {{\cal I}}} \mu\left( [-j,j] \cap (U_{j^*,k^*}-\Lambda)\setminus I_{j^*,k^*}  \right)
\end{equation}
$$\le \sum_{\substack{(j^*,k^*)\in {{\cal I}} \\ j^*<j}} \mu\left( [-j,j] \cap (U_{j^*,k^*}-\Lambda)\setminus I_{j^*,k^*}  \right) $$
$$+ \sum_{\substack{(j^*,k^*)\in {{\cal I}} \\}} \mu\left( [-j^*,j^*] \cap (U_{j^*,k^*}-\Lambda)\setminus I_{j^*,k^*}  \right)$$
$$\le \sum_{\substack{(j^*,k^*)\in {{\cal I}} \\ j^*<j}} 2j + 
\sum_{\substack{(j^*,k^*)\in {{\cal I}} \\}} (4j^*+3)\cdot E_{j^*,k^*}
$$ 
$$\leq  2j\cdot 2j(2^{j-1}+...+1)  + \sum_{j^*=1}^{\infty} \sum_{k^*=0}^{2j^*\cdot 2^{j^*}-1} 
(4j^*+3)E_{j^*,k^*}
$$ 
$$
\le 4j^2\cdot  
2^{j} + \sum_{j^*=1}^{\infty} 2j^*\cdot 2^{j^*} (4j^*+3)2^{-2j^*\cdot 2^{j^*}} 
$$
$$
\le 4j^2 \cdot 2^{j} + \sum_{j^*=1}^{\infty} \left(8(j^*)^2+6j^*\right)2^{-2j^*\cdot 2^{j^*}+j^*} < \infty,
$$
which by the Borel–Cantelli lemma implies \eqref{kompl}.

Let $G$ be a fixed open subset of $ {\ensuremath {\mathbb R}}$. If $x\in G$, then $\{(j,k)\in {{\cal I}} : x\in I_{j,k}\subset G \}$ is an infinite set, hence according to \eqref{I_{j,k} fedese} and \eqref{def f_G}
\[
\sum_{n=0}^\infty f_G(x+\lambda_n) = \infty.
\]
If $x\in {\mathbb {R}} {\setminus} G$ and $\sum_{n=0}^\infty  f_G(x+\lambda_n) =\infty$, then $\{n\in {\mathbb {N}} : x+\lambda_n\in U_G\}$ is an infinite set, which implies that $\{(j^*,k^*)\in {{\cal I}} :  I_{j^*,k^*}\subset G \text{    and    } x\in (U_{j^*,k^*}-\Lambda)\}$ is also infinite, thus \eqref{kompl} implies \eqref{tetelbeli}.

Next we see how one can modify $f_{G}$ to obtain a 
$g_{G}\in C_{0}^{+}( {\ensuremath {\mathbb R}})$ still satisfying 
\eqref{tetelbeli} and \eqref{*conv}. In \cite{[BKM1]} there is Proposition 1, which says
that one can modify $f_{G}$ to obtain a $g_{G}\in C_{0}^{+}( {\ensuremath {\mathbb R}})$ such that
$C(f_{G},\lambda)=C(g_{G},\lambda)$ a.e. and
$D(f_{G},\lambda)=D(g_{G},\lambda)$  a.e. Since we want to preserve 
\eqref{*conv} we cannot change $D(f_{G},\lambda)$ by an arbitrary set of measure zero. Hence in the next construction a little extra care is needed.

\begin{equation}\label{*defLN}
\text{   Put $ {\Lambda   }_N=\{{\lambda}\in {\Lambda}:  {\lambda}\leq 10N \}$
and $L_{N}=\# {\Lambda}_N$.}
\end{equation}

Observe that $U_{G}\cap (- {\infty},0]= {\emptyset}$, $U_{G}$ does not contain a half-line, and $U_{G}\cap [0,N]$ is the union of finitely
many disjoint closed intervals for any $N\in {\ensuremath {\mathbb N}}$.

Choose an open 
$\widetilde{U}_G\supset U_{G}$ such that it does not contain a half-line, and 
\begin{equation}\label{*UtG}
 {\mu}((\widetilde{U}_G {\setminus} U_{G})\cap[N-1,N])< \frac{2^{-N}}{L_{N}} \text{   for any   }N\in {\ensuremath {\mathbb N}}.
\end{equation}

Select a continuous function $\widetilde{g}_{G}$  such that $\widetilde{g}_{G}(x)=f_{G}(x)$ for $x\in U_{G}$,
$\widetilde{g}_{G}(x)=0$ if $x\not\in \widetilde{U}_G$ and $|\widetilde{g}_{G}|\leq 1.$
Hence $\widetilde{g}_{G}\geq f_{G}$ on $ {\ensuremath {\mathbb R}}$,  and
$D(\widetilde{g}_{G}, {\Lambda})\supset D(f_{G}, {\Lambda})\supset G.$

It is also clear that
$0\leq \widetilde{g}_{G}-f_{G}\leq \chi_{\widetilde{U}_G  {\setminus} U_{G}}=:h_{G}$, and
\begin{equation}\label{*A2*h6}
\sum_{{\lambda}\in  {\Lambda}}\Big (\widetilde{g}_{G}(x+ {\lambda})-f_{G}(x+ {\lambda})\Big )\leq
\sum_{{\lambda}\in  {\Lambda}}h_{G}(x+ {\lambda}).
\end{equation}

Next we prove  that 
\begin{equation}\label{*sumhgl}
\text{   $\sum_{{\lambda  }\in  {\Lambda }}h_{G}(x+ {\lambda})$ is finite almost everywhere,}
\end{equation}
yielding that $C(\widetilde{g}_{G}, {\Lambda})$ equals $C(f_{G}, {\Lambda})$ modulo a set of measure zero.

Put $H_{G,K, {\infty}}=\{x\in [-K,K]: \sum_{{\lambda}\in  {\Lambda}}h_{G}(x+ {\lambda})= {\infty} \}$.
We will show that 
\begin{equation}\label{*mest*}
\text{   for any $K>1$ we have $ {\mu   }(H_{G,K, {\infty}})=0$.}
\end{equation}
This clearly implies \eqref{*sumhgl}.

Observe that if $x\in H_{G,K, {\infty}}$, then there are infinitely many $ {\lambda}$s  such that 
$x+ {\lambda}\in \widetilde{U}_G {\setminus} U_{G}$, that is,
$x\in ((\widetilde{U}_G {\setminus} U_{G})- {\lambda})\cap [-K,K]$.
Thus, by the Borel--Cantelli lemma to prove \eqref{*mest*} it is sufficient to show
that
\begin{equation}\label{*BC*}
\sum_{{\lambda}\in {\Lambda}} {\mu}\Big (\Big ((\widetilde{U}_G {\setminus} U_{G})- {\lambda}\Big )\cap [-K,K]\Big )< {\infty}.
\end{equation}

This is shown by the following estimate
$$
\sum_{{\lambda}\in {\Lambda}} {\mu}\Big (\Big ((\widetilde{U}_G {\setminus} U_{G})- {\lambda}\Big )\cap [-K,K]\Big )=
\sum_{{\lambda}\in {\Lambda}}\sum_{N=1}^{{\infty}} {\mu}\Big (\Big (((\widetilde{U}_G {\setminus} U_{G})\cap[N-1,N])- {\lambda}\Big )\cap [-K,K]\Big )
$$
$$
=\sum_{N=1}^{{\infty}}\sum_{{\lambda}\in {\Lambda}} {\mu}\Big (\Big ((\widetilde{U}_G {\setminus} U_{G})\cap[N-1,N]\Big )\cap [ {\lambda}-K, {\lambda}+K]\Big )$$
$$
=\sum_{N=1}^{K}\sum_{{\lambda}\in {\Lambda}} {\mu}\Big (\Big ((\widetilde{U}_G {\setminus} U_{G})\cap[N-1,N]\Big )\cap [ {\lambda}-K, {\lambda}+K]\Big )
$$
$$
{\hspace*{2cm}}+\sum_{N=K+1}^{{\infty}}\sum_{{\lambda}\in {\Lambda}} {\mu}\Big (\Big ((\widetilde{U}_G {\setminus} U_{G})\cap[N-1,N]\Big )\cap [ {\lambda}-K, {\lambda}+K]\Big )
$$
(with a finite $S_{1}$)
$$
=S_1+\sum_{N=K+1}^{{\infty}}\sum_{\ \  {\lambda}\in {\Lambda},\,   {\lambda}\leq 10N} {\mu}\Big (\Big ((\widetilde{U}_G {\setminus} U_{G})\cap[N-1,N]\Big )\cap [ {\lambda}-K, {\lambda}+K]\Big )
$$
(now using \eqref{*defLN} and  \eqref{*UtG})
$$
\leq S_{1}+\sum_{N=K+1}^{{\infty}} L_{N}\cdot \frac{2^{-N}}{L_{N}}< {\infty}.
$$

So far we have shown that $\widetilde{g}_{G}$ satisfies \eqref{tetelbeli} and \eqref{*conv}.
Since $\widetilde{g}_{G}\in C^{+}( {\ensuremath {\mathbb R}})$, but not in $C_{0}^{+}( {\ensuremath {\mathbb R}})$. We need to adjust it a little 
further.

Since $G$ is open choose an increasing 
 sequence of compact sets $G_{K} {\subset} G\cap [-K,K]$
 such that $ {\bigcup}_{K=1}^{{\infty}}G_{K}=G.$
 
 Put $M_{0}=0$.
 Choose $M_{1}\in  {\ensuremath {\mathbb R}} $
 such that for any $x\in G_{1}$ we have 
 $$\sum_{{\lambda}\in {\Lambda}, \  M_{0}+10< {\lambda}<M_{1}}\widetilde{g}_{G}(x+ {\lambda})>1,$$
 and $\widetilde{g}_{G}(M_{1}+5)=0$. This latter property can be satisfied since 
 by assumption $\widetilde{U}_G$ does not contain a half-line.
 
 In general, if we already have selected $M_{K-1}$
 such that $\widetilde{g}_{G}(M_{K-1}+5(K-1))=0$ then choose $M_{K}\in  {\ensuremath {\mathbb R}}$
  such that  for any $x\in G_{K}$ we have
  \begin{equation}\label{*A5K*}
  \sum_{{\lambda}\in {\Lambda}, \  M_{K-1}+10K< {\lambda}<M_{K}}\widetilde{g}_{G}(x+ {\lambda})>K,
  \end{equation}
  and $\widetilde{g}_{G}(M_{K}+5K)=0$.

For $x\leq M_{1}+5$ we put $g_{G}(x)=\widetilde{g}_{G}(x)$. For $K>1$ and $x\in (M_{K-1}+5(K-1),
M_{K}+5K]$ we put $g_{G}(x)=\frac{1}{K}\widetilde{g}_{G}(x).$

It is clear that $g_{G}\in C_{0}^{+}( {\ensuremath {\mathbb R}}).$

Since $g_{G}\leq \widetilde{g}_{G}$ we have
$C(g_{G}, {\Lambda})\supset C(\widetilde{g}_{G}, {\Lambda}).$
If we can show that $G {\subset} D(g_{G}, {\Lambda})$ then we are done.
Suppose $x\in G$. Then there is a $K_{x}$
 such that $x\in G_{K}$ for any $K\geq K_{x}$.
 Therefore, for these $K$ we have $x\in [-K_{x},K_{x}] {\subset}
 [-K,K]$ and by using \eqref{*A5K*}
 $$
 \sum_{{\lambda}\in {\Lambda}, \  M_{K-1}+6K< {\lambda}<M_{K}+4K}g_{G}(x+ {\lambda})
 =\sum_{{\lambda}\in {\Lambda}, \  M_{K-1}+6K< {\lambda}<M_{K}+4K}\frac{1}{K} {{\widetilde {g}}}_{G}(x+ {\lambda})>1,
 $$
 for any $K\geq K_{x}$ and hence $x\in D(g_{G}, {\Lambda})$.
\end{proof}

\section{Subintervals in $C(f, {\Lambda})$}\label{*secsubi}

\begin{theorem}\label{*thdivG}
There exists an asymptotically dense infinite discrete set $ {\Lambda}$
 such that for any open set $G {\subset}  {\ensuremath {\mathbb R}}$ one can select an
 $f_{G}\in C^{+}_{0}( {\ensuremath {\mathbb R}})$  such that $D(f, {\Lambda})=G.$
\end{theorem}

\begin{remark}
 As Theorem 
\ref{*thcintl} shows in the above theorem we 
cannot assume that
$ {\Lambda}$ is a decreasing gap set. On the other hand, in our claim
we have $D(f, {\Lambda})=G$, that is,  there is no exceptional set of measure
zero where we do not know what happens. This also implies that if the interior
of $ {\ensuremath {\mathbb R}} {\setminus} G$ is non-empty then $C(f, {\Lambda})$ contains intervals.
\end{remark}

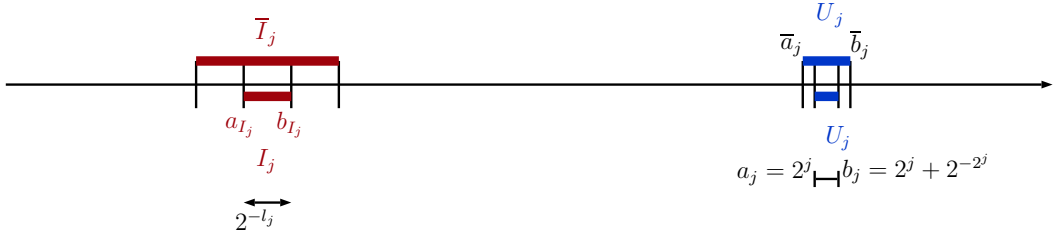
\begin{figure}[h]
\centering{
\resizebox{1.0\textwidth}{!}{
%
\psscalebox{1.0 1.0} 
{
\begin{pspicture}(0,-1.9084375)(18.38,1.9084375)
\definecolor{colour0}{rgb}{0.62352943,0.007843138,0.043137256}
\definecolor{colour1}{rgb}{0.0,0.21176471,0.7882353}
\psline[linecolor=black, linewidth=0.04, arrowsize=0.05291667cm 2.0,arrowlength=1.4,arrowinset=0.0]{->}(0.0,0.65490234)(17.6,0.65490234)
\psline[linecolor=black, linewidth=0.04](3.2,1.0549023)(3.2,0.25490233)
\psline[linecolor=black, linewidth=0.04](4.0,1.0549023)(4.0,0.25490233)
\psline[linecolor=black, linewidth=0.04](4.8,1.0549023)(4.8,0.25490233)
\psline[linecolor=black, linewidth=0.04](5.6,1.0549023)(5.6,0.25490233)
\psline[linecolor=black, linewidth=0.04](13.6,1.0549023)(13.6,0.25490233)
\psline[linecolor=colour0, linewidth=0.16](4.0,0.45490235)(4.8,0.45490235)
\psline[linecolor=colour0, linewidth=0.16](3.2,1.0549023)(5.6,1.0549023)
\rput[t](4.78,0.19490235){\textcolor{colour0}{$b_{I_j}$}}
\rput[t](3.94,0.094902344){\textcolor{colour0}{$a_{I_j}$}}
\rput[t](4.38,-0.44509766){\textcolor{colour0}{$I_{j}$}}
\rput[t](4.38,1.7949023){\textcolor{colour0}{$\overline{I}_{j}$}}
\rput[t](4.2,-1.4650977){$2^{-l_j}$}
\psline[linecolor=black, linewidth=0.04, arrowsize=0.05291667cm 2.0,arrowlength=1.4,arrowinset=0.0]{<->}(4.0,-1.3450977)(4.8,-1.3450977)
\psline[linecolor=black, linewidth=0.04](14.0,1.0549023)(14.0,0.25490233)
\psline[linecolor=black, linewidth=0.04](13.4,1.0549023)(13.4,0.25490233)
\psline[linecolor=black, linewidth=0.04](14.2,1.0549023)(14.2,0.25490233)
\psline[linecolor=colour1, linewidth=0.16](13.4,1.0549023)(14.2,1.0549023)
\psline[linecolor=colour1, linewidth=0.16](13.6,0.45490235)(14.0,0.45490235)
\rput[t](14.0,-0.045097657){\textcolor{colour1}{$U_{j}$}}
\rput[b](13.86,1.6349024){\textcolor{colour1}{$\overline{U}_{j}$}}
\rput[r](13.58,-0.8250977){$a_{j}=2^j$}
\rput[l](14.06,-0.7650977){$b_{j}=2^j+2^{-2^j}$}
\rput[r](13.4,1.3349023){$\overline{a}_{j}$}
\rput[l](14.2,1.3349023){$\overline{b}_{j}$}
\psline[linecolor=black, linewidth=0.04, tbarsize=0.07055555cm 5.1]{|*-|*}(13.6,-0.9450977)(14.0,-0.9450977)
\end{pspicture}
}}}
\caption{Definition of $I_{j}$, $U_{j}$ and related sets} \label{*fig2}
\end{figure}

\begin{proof}
Denote by $ {{\cal I}}_{D}=\{[(k-1)/2^{l},k/2^{l}] :\  k,l\in {\ensuremath {\mathbb Z}}, \  l\geq 0 \}$
the system of dyadic intervals. It is clear that one can enumerate the elements
of $ {{\cal I}}_{D}$ in a sequence $\{I_{j} \}_{j=1}^{{\infty}}$ which satisfies the following properties
\begin{equation}\label{*Ija}
I_{j}=[a_{I_j},b_{I_j}]=\Big [\frac{k_{j}-1}{2^{l_{j}}},\frac{k_{j}}{2^{l_{j}}}\Big ] {\subset} [-j,j] 
\text{   and   }  {\mu}(I_{j})=2^{-l_{j}}\geq \frac{1}{j}.
\end{equation}
We denote by $\overline{I}_j$ the closed interval which is concentric with $I_{j}$ but is of length three times the length of $I_{j}$.

We put $$U_{j}=[a_{j},b_{j}]=[2^j,2^j+2^{-2^j}]\text{   and   }\overline{U}_j=[a_{j}-2^{-2^j-j-1},b_{j}+2^{-2^j-j-1}]=
[\overline{a}_{j},\overline{b}_{j}]. $$
See Figure \ref{*fig2}.

We suppose that $f_{j}(x)=0$ if $x\not \in \overline{U}_j$, $f_{j}(x)=2^{-j}$ if $x \in U_j$, the function $f_{j}$  is continuous on $ {\ensuremath {\mathbb R}}$ and is linear on the connected components of $\overline{U}_j {\setminus} U_{j}.$
We define
\begin{equation}\label{*Lojd}
 {\Lambda}_{1,j}=\{k\cdot 2^{-2^j-j}:k\in {\ensuremath {\mathbb Z}} \}\cap [2^j -k_{j}2^{-l_j},2^j+2^{-2^j}-(k_{j}-1)2^{-l_j}]
\end{equation}
$$=\{k\cdot 2^{-2^j-j}:k\in {\ensuremath {\mathbb Z}} \}\cap [a_{j}-b_{I_j},b_{j}-a_{I_j}]$$
and put $  \displaystyle   {\Lambda}_{1}= {\bigcup}_{j=1}^{{\infty}} {\Lambda}_{1,j}.$

Observe that if $x\in I_{j}$ then
\[
x+\min\Lambda_{1,j} \le b_{I_j}+\min\Lambda_{1,j} = b_{I_j}+a_j-b_{I_j} = a_j
\]
and
\[
x+\max\Lambda_{1,j} \ge a_{I_j}+\max\Lambda_{1,j} = a_{I_j}+b_j-a_{I_j} = b_j,
\]
hence
\begin{equation}\label{*L1j}
\sum_{{\lambda}\in {\Lambda}_{1,j}}f_{j}(x+ {\lambda})\geq \frac{\diam U_j}{2^{-2^j-j}} 2^{-j} = \frac{2^{-2^j}}{2^{-2^j-j}} 2^{-j} = 1.
\end{equation}
On the other hand, by \eqref{*Ija}
\begin{equation*}
\begin{split}
\overline{U}_j-\Lambda_{1,j} &= \left[\min\overline{U}_j-\max\Lambda_{1,j},\max\overline{U}_j-\min\Lambda_{1,j}\right] \\
= [\overline{a}_j-b_j+&a_{I_j},\overline{b}_j-a_j+b_{I_j}] = 
\left[a_{I_j}-2^{-2j}-2^{-2^{j}-j-1},b_{I_j}+2^{-2j}+2^{-2^{j}-j-1}\right] \\
&\subset \left[a_{I_j}-\frac{1}{j},b_{I_j}+\frac{1}{j}\right]  \subset \left[a_{I_j}-2^{-l_j},b_{I_j}+2^{-l_j}\right] = \overline{I}_j
\end{split}
\end{equation*}
thus
\begin{equation}\label{*L1jo}
\sum_{{\lambda}\in {\Lambda}_{1,j}}f_{j}(x+ {\lambda})=0\text{   if $x\in [-j,j]$, $x\not\in \overline{I   }_j$.}
\end{equation}

Suppose $G {\subset}  {\ensuremath {\mathbb R}}$ is a given open set and put $ {{\cal J}}_{G}=\{j:\overline{I}_j {\subset} G \}.$
Let $f_{G}(x)=\sum_{j\in {{\cal J}}_{G}}f_{j}(x).$
Then $f_{G}$ is continuous and non-negative on $ {\ensuremath {\mathbb R}}$ and clearly
$\lim_{x\to {\infty}}f(x)=0.$

We claim that
\begin{equation}\label{*DL1}
\sum_{{\lambda}\in {\Lambda}_{1}} f_{G}(x+ {\lambda})=+ {\infty} 
\end{equation}
exactly on $G$.

Indeed, if $x\in G$ then there are infinitely many $j$s  such that $x\in I_{j} {\subset} 
\overline{I}_j {\subset} G$.
This means that \eqref{*L1j} holds for infinitely many $j\in  {{\cal J}}_{G}$
and hence \eqref{*DL1} is true when $x\in G$.

Next we need to verify that \eqref{*DL1} does not hold for $x\not\in G$.
Suppose that $j_{0}\geq 10$, $j_{0}\in  {{\cal J}}_{G}$, $x\not\in G$ and
$x\in [-j_{0},j_{0}]$. Then $x\not \in \overline{I}_{j_0}$ and by \eqref{*L1jo} we have
\begin{equation}\label{*L1jo3}
\sum_{{\lambda}\in  {\Lambda}_{1,j_0}}f_{j_{0}}(x+ {\lambda})=0.
\end{equation}

Next assume that $j<j_{0}$. Then by using \eqref{*Ija} and \eqref{*Lojd}
$$\max \{x+ {\lambda}:\  {\lambda}\in  {\Lambda}_{1,j} \}\leq j_{0}+2^j+2^{-2^j}-(k_{j}-1)2^{-l_j}\leq j_{0}+2^j+2^{-2^j}+j$$
$$< 2j_{0}+2^{j_{0}-1}+1<2^{j_0}-1<2^{j_0}-2^{-2^{j_{0}}-j_{0}-1}=\overline{a}_{j_0}.$$
Hence,
\begin{equation}\label{*L1jo4}
\sum_{{\lambda}\in {\Lambda}_{1,j}}f_{j_{0}}(x+ {\lambda})=0.
\end{equation}

If $j_{0}<j$ then 
$$\min\{x+ {\lambda}:\  {\lambda}\in {\Lambda}_{1,j} \}\geq -j_{0}+2^j-j>
2^{j-1}-2j-1+2^{j-1}+1>
2^{j_0}+1>\overline{b}_{j_0},$$
and hence
in this case we also have \eqref{*L1jo4}.

Therefore, from \eqref{*L1jo3} and \eqref{*L1jo4} it follows that
\begin{equation}\label{*L1jo5}
\sum_{{\lambda}\in {\Lambda}_{1}}f_{j_{0}}(x+ {\lambda})=0 \text{   for   }j_{0}\in  {{\cal J}}_{G},\  j_{0}\geq 10,\  |x|\leq j_{0}.
\end{equation}

This implies
\begin{equation*}\label{*L1jo6}
\sum_{{\lambda}\in {\Lambda}_1}f_{G}(x+ {\lambda})\leq \sum_{\substack{{\lambda}\in {\Lambda}_{1,j} \\ j\leq \max\{10,|x| \}}}
f_{j}(x+ {\lambda})<+ {\infty}.
\end{equation*}

Since $ {\Lambda}_{1}$ is not asymptotically dense we need to choose an
asymptotically dense   $ {\Lambda}_{2}$ such that 
\begin{equation}\label{*DL2}
\sum_{{\lambda}\in {\Lambda}_{2}}\sum_{j=1}^{{\infty}}f_{j}(x+ {\lambda})<+ {\infty}\text{   holds for any   }x\in  {\ensuremath {\mathbb R}}.
\end{equation}
Then for any open $G {\subset}  {\ensuremath {\mathbb R}}$
$$
\sum_{{\lambda}\in {\Lambda}_{2}}f_G(x+ {\lambda})\leq \sum_{{\lambda}\in {\Lambda}_{2}}\sum_{j=1}^{{\infty}}f_{j}(x+ {\lambda})<+ {\infty}
$$
holds and if we let $ {\Lambda}= {\Lambda}_{1}\cup {\Lambda}_{2}$ then $ {\Lambda}$ is asymptotically dense
and $D(f_{G}, {\Lambda})=G$.

To complete the proof of this theorem we need to verify \eqref{*DL2}
for a suitable $ {\Lambda}_{2}$.
For $j\geq 10$ put
$$ {\Lambda}_{2,j}=\{k\cdot 2^{-j} :k\in {\ensuremath {\mathbb Z}} \}\cap (2^{j-1}+2(j-1),2^j+2j]
,\text{   and   } {\Lambda}_{2}= {\bigcup}_{j=10}^{{\infty}} {\Lambda}_{2,j}.$$

Suppose $x\in [-j_{0},j_{0}]$ and $j_{0}\geq 10$. Then for
$j\geq j_{0}$ from $x+ {\lambda}\in \overline{U}_j$ it follows that
$2^j-1< x+ {\lambda}\leq j+ {\lambda}$, and hence
$$ {\lambda}>2^j-j-1>2^{j-1}+2(j-1).$$
Similarly, $x+ {\lambda}\in \overline{U}_j$ implies
$2^j+1>x+ {\lambda}\geq  -j+ {\lambda}$, and hence
$$ {\lambda}<2^j+j+1<2^j+2j.$$
Thus from $x+ {\lambda}\in \overline{U}_j$ it follows that $ {\lambda}\in  {\Lambda}_{2,j}$.
Since the length of $\overline{U}_j$ is less than $2\cdot 2^{-2^j}<2^{-j}$
there is at most one $ {\lambda}\in {\Lambda}_{2,j}$ for which $f_{j}(x+ {\lambda})\not=0$
and for this $ {\lambda}$ we have $f_{j}(x+ {\lambda})=2^{-j}.$

Put $M_{x}=\max\{10,|x| \}.$
Then
$$\sum_{{\lambda}\in {\Lambda}_{2}}\sum_{j=1}^{{\infty}}f_{j}(x+ {\lambda})=\sum_{{\lambda}\in {\Lambda}_{2}}\sum_{j=1}^{M_{x}}f_{j}(x+ {\lambda})+\sum_{j=M_{x}+1}^{{\infty}}\sum_{{\lambda}\in {\Lambda}_{2}}f_{j}(x+ {\lambda})$$
$$
 \le \sum_{{\lambda}\in {\Lambda}_{2}}\sum_{j=1}^{M_{x}}f_{j}(x+ {\lambda})+\sum_{j=M_{x}+1}^{{\infty}}2^{-j}<+ {\infty}.
$$

\end{proof}

In Theorem \ref{*thdflg}   we verified that for decreasing gap asymptotically dense sets $D(f, {\Lambda})$ can contain an open set, while 
$C(f, {\Lambda})$ equals the complement of this 
open set only almost everywhere.

The next example shows that one can define decreasing gap asymptotically dense
$ {\Lambda}$s for which one can find nonnegative continuous $f$s such that both
$C(f, {\Lambda})$ and $D(f, {\Lambda})$  have interior points.

\begin{theorem}\label{*excint}
There exists a decreasing gap asymptotically dense $ {\Lambda}$ and an $f\in C^{+}_{0}( {\ensuremath {\mathbb R}})$ such that
$I_{1}=[0,1] {\subset} D(f, {\Lambda})$ and $I_{2}=[4,5] {\subset} C(f, {\Lambda})$.
\end{theorem} 

\begin{proof}
Put $f(x)=2^{-2^{j+1}}$ if $x\in[10j,10j+1]$ for a $j\in {\ensuremath {\mathbb N}}$.
Set $f(x)=0$ if $x\in \{10j-1/4,10j+5/4 \}$ for a $j\in {\ensuremath {\mathbb N}}$,
and also put $f(x)=0$ for $x\leq 0$.
We suppose that $f$ is linear on the intervals where we have not defined it so far.
Put $ {\Lambda}_{1,j}=\{k\cdot 2^{-2^{j}} :\ k\in {\ensuremath {\mathbb Z}}\}\cap [10j-10, 10j-2)$ and
$ {\Lambda}_{2,j}=\{k\cdot 2^{-2^{j+1}} :\ k\in {\ensuremath {\mathbb Z}} \}\cap [10j-2, 10j)$. Let $\Lambda =  {\bigcup}_{j=1}^\infty (\Lambda_{1,j}\cup\Lambda_{2,j})$. Observe that $\Lambda$ is 
 a decreasing gap asymptotically dense set.

One can see that for $x\in I_{1}$ we have 
$$\sum_{{\lambda}\in {\Lambda}}f(x+ {\lambda})\geq\sum_{j=1}^{{\infty}}2^{2^{j+1}}\cdot 2^{-2^{j+1}}=+ {\infty}$$
and for $x\in I_{2}$
$$\sum_{{\lambda}\in {\Lambda}}f(x+ {\lambda})\leq\sum_{j=1}^{{\infty}}2\cdot 2^{2^j}\cdot 2^{-2^{j+1}}<+ {\infty}.$$
It is also clear from the construction that $\lim_{x\to {\infty}}f(x)=0.$
\end{proof}

Observe that in the above construction $I_{1} {\subset} D(f, {\Lambda})$ was to the left of $I_{2} {\subset} C(f, {\Lambda})$.
 The next theorem shows that for decreasing gap asymptotically dense $ {\Lambda}$s and continuous functions this situation cannot be improved. 
If $x$ is an interior point of $C(f, {\Lambda})$ then the half-line $[x, {\infty})$ intersects $D(f, {\Lambda})$ in a set of measure zero. 
As Theorem \ref{*thdivG} shows if we do not assume that $ {\Lambda}$ is of decreasing gap then
it is possible that $D(f, {\Lambda})$ has  a part of positive measure, even to the 
right of the  interior points of $C(f, {\Lambda})$.

\begin{theorem}\label{*thcintl} Let $\Lambda$ be a decreasing gap and asymptotically dense set, and let $f: {\ensuremath {\mathbb R}}\to[0,+ {\infty})$ be continuous. Then if $x$ is an interior point of $C(f,\Lambda)$ then 
\begin{equation}\label{*zhalf}
 {\mu}\Big ([x,+ {\infty})\cap D(f, {\Lambda})\Big )=0.
\end{equation}
 \end{theorem}

\begin{proof} Proceeding towards a contradiction assume the existence of a 
non-degenerate closed interval $I\subset C(f,\Lambda)$. Suppose that there is a bounded subset $D_1(f,\Lambda) {\subset} D(f,\Lambda)$  with positive measure
to the right of $I$. 
Choose an interval $J=[a_{J},b_{J}]$ to the right of $I$ such that 
\begin{equation}\label{*chJ}
 {\mu}(J)= {\mu}(I)/10, \text{   and   } {\mu}(J\cap D(f, {\Lambda}))= {\alpha} >0.
\end{equation}
We put $D_1(f, {\Lambda}) =J\cap D(f, {\Lambda})$.
We suppose that $ {\Lambda}=\{{\lambda}_{1}, {\lambda}_{2},... \}$ is indexed in an increasing order.
Select $N$  such that 
\begin{equation}\label{*Nch}
\text{   $ {\lambda  }_{n }- {\lambda}_{n-1}<\frac{{\mu}(I)}{100}$ for $n\geq N$.}
\end{equation}

We clearly have that $\sum_{i=N}^{\infty}f(x+\lambda_i)$ diverges on $D_1(f,\Lambda)$.
 Moreover, if $n\in\mathbb{N}$, which is to be fixed later, for large enough $M$ we have $\sum_{i=N}^{M}f(x+\lambda_i)>n$ in a set $D_2(f,\Lambda)\subset D_1(f,\Lambda)$ of measure larger than $\frac{{\alpha}}{2}$. Hence we have
\begin{equation}\label{*Dtint}
\int_{D_2(f,\Lambda)}\sum_{i=N}^{M}f(x+\lambda_i)dx\geq \frac{n {\alpha}}{2}.
\end{equation}

Assume that  $i\in \{N,N+1,...,M\}$.  We choose $ {\gamma}(i)$ such that
\begin{equation}\label{*defjxi}
a_{J}+ {\lambda}_{i}- {\lambda}_{{\gamma}(i)}\in I, \text{   but   } a_{J}+ {\lambda}_{i}- {\lambda}_{{\gamma}(i)+1}\not\in I.
\end{equation}

Since  $a_{J}$ is to the right of
$I$ it is clear that 
$ {\lambda}_{{\gamma}(i)}> {\lambda}_{i}$, therefore
$ {\gamma}(i)>i\geq N$ and hence \eqref{*Nch} implies that
$ {\gamma}(i)$ is well-defined, that is \eqref{*defjxi} can be satisfied.

It is also clear that  there exists $\widetilde{M}$  such that $ {\gamma}(i)\leq \widetilde{M}$
holds for $i\in\{N,N+1,...,M\}$.

By \eqref{*chJ}, \eqref{*Nch},  and  \eqref{*defjxi}  we have 
\begin{equation}\label{*Jigic}
J+ {\lambda}_{i}- {\lambda}_{{\gamma}(i)} {\subset} I\text{   and hence   }D_2(f,\Lambda)+ {\lambda}_{i}- {\lambda}_{{\gamma}(i)}
 {\subset} I.
\end{equation}

Next we verify that 
\begin{equation}\label{*ivi}
\text{   if $i'\not=i$ then $ {\gamma  }(i')\not= {\gamma}(i)$. }
\end{equation}

Indeed, we can suppose that $i'<i$, and proceeding towards a contradiction
we also suppose that $ {\gamma}(i')= {\gamma}(i)$.
We know that $a_{J}+ {\lambda}_{i}- {\lambda}_{{\gamma}(i)}\in I$, moreover $a_{J}+ {\lambda}_{i'}- {\lambda}_{{\gamma}(i')}\in I$
holds as well. Since $ {\gamma}(i)= {\gamma}(i')$ we have 
$$a_{J}+ {\lambda}_{i'}- {\lambda}_{{\gamma}(i')}=a_{J}+ {\lambda}_{i}- {\lambda}_{{\gamma}(i)}- {\lambda}_{i}+ {\lambda}_{i'}\in I.$$

Using the first half of \eqref{*defjxi} and $ {\lambda}_{i'}\leq  {\lambda}_{i-1}< {\lambda}_{i}$
we also obtain
$$a_{J}+ {\lambda}_{i}- {\lambda}_{{\gamma}(i)}- {\lambda}_{i}+ {\lambda}_{i'}\leq a_{J}+ {\lambda}_{i}- {\lambda}_{{\gamma}(i)}- {\lambda}_{i}+ {\lambda}_{i-1}\in I.$$

Since $ {\Lambda}$ is of decreasing gap and $ {\gamma}(i)>i$ we have $ {\lambda}_{{\gamma}(i) +1}- {\lambda}_{{\gamma}(i)}<
 {\lambda}_{i}- {\lambda}_{i-1}$, and hence
$$a_{J}+ {\lambda}_{i}- {\lambda}_{{\gamma}(i)}- {\lambda}_{i}+ {\lambda}_{i-1}< a_{J}+ {\lambda}_{i}- {\lambda}_{{\gamma}(i)}- {\lambda}_{{\gamma}(i)+1}+ {\lambda}_{{\gamma}(i)}\in I,$$
which contradicts \eqref{*defjxi}.

By using \eqref{*Jigic} and  \eqref{*ivi} we infer 
\begin{equation}\label{*sumint}
\int_{D_2(f,\Lambda)}  \sum_{i=N}^{M} f(x+ {\lambda}_{i})dx=  \sum_{i=N}^{M} \int_{D_2(f,\Lambda)}
 f(x+ {\lambda}_{i}- {\lambda}_{{\gamma}(i)}+ {\lambda}_{{\gamma}(i)})dx
\end{equation}
$$=  \sum_{i=N}^{M} \int_{D_2(f,\Lambda)+ {\lambda}_{i}- {\lambda}_{{\gamma}(i)}}  f(t+ {\lambda}_{{\gamma}(i)})dt\leq
\int_{I}\sum_{j=N}^{\widetilde{M}}f(t+ {\lambda}_{j})dt.$$

 Thus by \eqref{*Dtint} we obtain
\begin{displaymath}
\int_{I}\sum_{i=N}^{\widetilde{M}}f(x+\lambda_i)dx\geq \frac{n {\alpha}}{2},
\end{displaymath}
as the left-handside by \eqref{*sumint} gives an upper bound for 
the integral in \eqref{*Dtint}. 
However, $\sum_{i=N}^{\widetilde{M}}f(x+\lambda_i)$ is continuous, 
which yields that this integrand is at least $\frac{n {\alpha}}{4\mu(I)}$ in a 
non-degenerate
closed subinterval $I_1\subset{I}$. 
Thus we have $s(x)=\sum_{\lambda\in\Lambda}f(x+\lambda)>\frac{n {\alpha}}{4\mu(I)}$ in $I_1$. 
Hence, if we choose $n$ to be large enough, we find that $s(x)>1$ in $I_1$. 

Now by applying the very same argument to $I_1$ instead of $I$, we might obtain that $s(x)> \frac{n_1 {\alpha}}{4\mu(I_1)}$ in a non-degenerate
 closed subinterval $I_2\subset{I_1}$. Thus if we choose $n_1$ to be large enough, we find that $s(x)>2$ in $I_2$. Proceeding recursively we obtain a nested sequence of closed intervals $I_1,I_2,...$ such that $s(x)>k$ for $x\in I_k$. As this system of intervals has a nonempty intersection, we find that there is a point in $I$ with $s(x)=\infty$, a contradiction. \end{proof}

\section{Acknowledgements}

During the Fall semester of 2018, when this paper was prepared all three authors visited 
the Institut Mittag-Leffler in Djursholm and participated in the 
semester Fractal Geometry and Dynamics. We thank the hospitality and financial support of the Institut Mittag-Leffler.
Z. Buczolich also thanks the R\'enyi Institute where he was
a visiting researcher for the academic year 2017-18.


\end{document}